\documentclass[reqno]{amsart}
\usepackage[T1]{fontenc}
\usepackage{latexsym,amssymb,amsmath,xcolor}

\usepackage[all,ps]{xy}
\usepackage{amsthm}
\usepackage{longtable}
\usepackage{epsfig}
\usepackage{mathrsfs}
\usepackage{hhline}
\usepackage{epic}
\usepackage{enumerate}
\usepackage{xcolor}
\usepackage{pgf,tikz}
\usepackage{mathrsfs}
\usetikzlibrary{arrows,shapes,calc}

\newcommand{\N}{\mathbb{N}}
 \newcommand{\Q}{\mathbb{Q}}
 \newcommand{\Z}{\mathbb{Z}}
 \newcommand{\R}{\mathbb{R}}
 
 \newcommand{\Sym}{\mathfrak{S}}

\newdimen\shadedBaseline\shadedBaseline=-4mm
\newcount\tableauRow\newcount\tableauCol
\newcommand\ShadedTableau[2][\relax]{%
  \begin{tikzpicture}[scale=0.4,draw/.append style={thick,black},baseline=\shadedBaseline]
    \ifx\relax#1\relax%
    \else 
      \foreach\bx in {#1} { \filldraw[blue!20]\bx+(-.5,-.5)rectangle++(.5,.5); }
    \fi
    \tableauRow=0
    \foreach \Row in {#2} {
       \tableauCol=1
       \foreach\k in \Row {
          \draw(\the\tableauCol,\the\tableauRow)+(-.5,-.5)rectangle++(.5,.5);
          \draw(\the\tableauCol,\the\tableauRow)node{\k};
          \global\advance\tableauCol by 1
       }
       \global\advance\tableauRow by -1
    }
  \end{tikzpicture}%
}
\newcommand\diag[3][\relax]{%
  \begin{tikzpicture}[scale=0.4,draw/.append style={thick,black},baseline=\shadedBaseline]
    \ifx\relax#1\relax%
    \else 
      \foreach\bx in {#1} { \filldraw[blue!20,dashed]\bx+(-.5,-.5)rectangle++(.5,.5); }
    \fi
    \tableauRow=0
    \foreach \Row in {#2} {
       \tableauCol=1
       \foreach\k in \Row {
          \draw(\the\tableauCol,\the\tableauRow)+(-.5,-.5)rectangle++(.5,.5);
          \draw(\the\tableauCol,\the\tableauRow)node{\k};
          \global\advance\tableauCol by 1
       }
       \global\advance\tableauRow by -1
    }
    \foreach \x in {#3} {
      \draw[red,dashed](\x+3,-2-\x)+(-.5,-.5)rectangle++(.5,.5);
    }
  \end{tikzpicture}%
}

\newcommand\frob[7][\relax]{%
  \begin{tikzpicture}[scale=0.4,draw/.append style={black},baseline=\shadedBaseline]
    \ifx\relax#1\relax%
    \else 
      \foreach\bx in {#1} { \filldraw[gray!20,dashed]\bx+(-.5,-.5)rectangle++(.5,.5); }
    \fi
    \tableauRow=0
    \foreach \Row in {#2} {
       \tableauCol=1
       \foreach\k in \Row {
          \draw(\the\tableauCol,\the\tableauRow)+(-.5,-.5)rectangle++(.5,.5);
          \draw(\the\tableauCol,\the\tableauRow)node{\k};
          \global\advance\tableauCol by 1
       }
       \global\advance\tableauRow by -1
    }
        \foreach \x in {#3} {
      \draw[gray!80,dashed](\x+3,-2-\x)+(-.5,-.5)rectangle++(.5,.5);
    }
    \foreach\bx in {#4} {
        \draw[thick]\bx+(-.3,0)rectangle++(.3,.0);
    }
    \foreach\bx in {#5} {
        \draw[thick]\bx+(0,.3)rectangle++(0,-.3);
    }
    \foreach\bx in {#6} {
        \draw[thick,densely dotted]\bx+(-.3,0)--++(.3,.0);
    }
    \foreach\bx in {#7} {
        \draw[thick,densely dotted]\bx+(0,.3)--++(0,-.3);
    }

  \end{tikzpicture}%
}

\newtheorem{theorem}{Theorem}[section] 
\newtheorem{lemma}[theorem]{Lemma}     

\newtheorem{example}[theorem]{Example}
\theoremstyle{definition}
\newtheorem{remark}[theorem]{Remark}

\title[]
{A crank-based approach to the theory of $3$-core partitions}
\author{Olivier Brunat}

\address{Universit\'e Paris-Diderot Paris 7\\ Institut de math\'ematiques de
         Jussieu -- Paris Rive Gauche\\ UFR de math\'e\-matiques\\ Case
7012\\ 75205 Paris Cedex 13\\
         France.}
\email{olivier.brunat@imj-prg.fr}

\author{Rishi Nath}

\address{York College, City University of New York, 
94--20 Guy R. Brewer Blvd. \\
Jamaica, NY 11435\\
USA
}
\email{rnath@york.cuny.edu}

\subjclass[2010]{05A17,\, 11P83}

\begin{document} 

\begin{abstract} 
    This note is concerned with the set of integral solutions of the
equation $x^2+3y^2=12n+4$, where $n$ is a positive integer. We will
describe a parametrization of this set using the $3$-core partitions of
$n$. In particular we construct a crank using the action of a suitable
subgroup of the isometric group of the plane that we connect with the unit
group of the ring of Eisenstein integers. We also show that the process
goes in the reverse direction: from the solutions of the equation and the
crank, we can describe the $3$-core partitions of $n$. As a consequence we
describe an explicit bijection between $3$-core partitions and ideals of
the ring of Eisenstein integers, explaining a result of G. Han and K. Ono
obtained using modular forms.
\end{abstract} 
\maketitle

\section{Introduction} 
\label{sec:intro}
    
    Let $t$ be a positive integer. The set of $t$-core partitions is a
subset of the set of integer partitions which is important in the
representation theory of the symmetric groups $\mathfrak S_n$. Indeed, the
famous Nakayama conjecture (since proven) says that, when $t$ is a prime
number, $t$-core partitions label the $t$-blocks of the symmetric groups.
In particular, defect zero characters of the symmetric group are labeled
by $t$-core partitions of $n$. On the other hand, $t$-core partitions also
appear crucially in number theory. For example, F. Garvan, D. Kim and D.
Stanton give in~\cite{GarvanKimStanton} a very elegant proof of
Ramanujan's congruences by constructing cranks using $t$-core partitions.
Following Dyson~\cite{dyson}, a crank over a finite set $A$ is a map
$\mathfrak c:A \rightarrow B$ to another set $B$, such that all fibers of
$\mathfrak c$ have the same cardinality. In this paper, we consider the
equation over integers
\begin{equation}
\label{eq:equationmain}
x^2+3y^2=k,
\end{equation}
where $k$ is a nonnegative integer. Denote by $\mathcal U_k$ the set of
tuples $(x,y)\in\Z^2$ that are solution of~(\ref{eq:equationmain}), and
set $u(k) = |\mathcal U_k|$. Equation~(\ref{eq:equationmain}) was first
considered by Fermat who conjectured in 1654~\cite[p.\,8]{cox} that if $k$
is a prime number congruent to $1$ modulo $3$, then $u(k) \neq 0$. This
conjecture was proved by Euler~\cite[p.\,12]{cox} in 1759. 

    In this note, we consider Equation~(\ref{eq:equationmain}) for $k =
12n+4$, where $n$ is a nonnegative integer. We will construct a crank
$\mathfrak c:\mathcal U_{12n+4}\rightarrow (\Z/3\Z)^2$ connected with the
set $\mathcal E_n$ of $3$-core partitions of $n$. In particular, we obtain
a parametrization of the elements of $\mathcal U_{12n+4}$ that depends on
$\mathcal E_n$. Moreover, this process goes in the reverse direction. From
the set of solutions of~(\ref{eq:equationmain}), we can recover $\mathcal
E_n$ as a fiber of $\mathfrak c$. 
    In fact, such a connection is not surprising. Indeed, in 2014, N.D.
Baruah and K. Nath showed in~\cite[Theorem 3.1]{BaruahNath} using the
theory of modular forms and Ramanujan's theta function identities that
$$u(12n+4)=6a_3(n),$$
where $a_3(n)=|\mathcal E_n|$. 
Note that as a consequence of our work, we obtain a new proof of this
equality.

    Recall that a partition $\lambda$ of $n$ is a non-increasing sequence
of positive integers summing to $n$. The integer $n$ is then called the
size of $\lambda$, also denoted by $|\lambda|$. Let $\mathcal P$ be the
set of partitions of integers, and $\mathcal P_n$ be the one of partitions
of $n$. To any $\lambda = (\lambda_1, \ldots, \lambda_r)\in\mathcal P$, we
associate its Young diagram. This is a collection of boxes arranged in
left- and top-aligned rows so that the number of boxes in its $j$th row is
$\lambda_j$. 

    The hook associated to a box $b$ of the Young diagram of $\lambda$
is the set of the boxes of the diagram to the right and below $b$. The
number of boxes appearing in a hook is called its hooklength. 
    We recall that a $3$-core partition is a partition with no
hooklength equal to $3$. We denote by $\mathcal E$ the set of $3$-core
partitions. Furthermore, as already mentioned, we write $\mathcal E_n$ for
the set of $\lambda\in\mathcal E$ of size $n$.
Recall that A. Granville and K. Ono showed in~\cite{GranvilleOno} that
every natural number has a $t$-core partition whenever $t\geq 4$. In the
same paper, the authors proved~\cite[\S3]{GranvilleOno} that this is not
the case for $t=3$; see also Remark~\ref{rk:t3}.

    The construction in this paper is based on the characteristic vectors
approach for the $3$-core partitions. In~\cite{GarvanKimStanton}, Garvan,
Kim and Stanton proved that the set $\mathcal E$ is in bijection with the
set 
\begin{equation}
\label{eq:enscar3}
\mathcal C=\{(c_0,c_1,c_2)\in\Z^3\mid c_0+c_1+c_2=0\}.
\end{equation}


    Recently, the two authors proposed a new approach based on the
Frobenius symbol of a partition~\cite{frobenius} to interpret this
bijection using abaci theory. Note that abaci methods were previously used
in~\cite{OnoSze} by K. Ono and L. Sze to connect $4$-core partitions and
class group structure of the integer rings of the imaginary quadratic
fields $\Q(\sqrt{-8n-5})$. Such links were also established in many
cases; see for example~\cite{levent},~\cite{BringmannKaneMales}
and~\cite{OnoRaji}.

    Now, we present more precisely our strategy. For any nonnegative
integer $n$, we define a crank $\mathfrak c:\mathcal U_{12n+4} \rightarrow
(\Z/3\Z)^2$. To prove that $\mathfrak c$ is a crank, we will describe a
natural action of $\Z/6\Z$ over $\mathcal U_{12n+4}$ and show that it
permutes the fibers of $\mathfrak c$; see Theorem~\ref{thm:action}. On the
other hand, we will construct an injective map between $\mathcal E_n$ and
$\mathcal U_{12n+4}$ (see Lemma~\ref{lemma:injection}), and we show that
its image is a fiber of $\mathfrak c$. In particular, this is a set of
representatives for the $\Z/6\Z$-action. We then obtain a way to
parametrize the elements of $\mathcal U_{12n+4}$. See
Theorem~\ref{thm:main}. Conversely, we can recover $\mathcal E_n$
from~$\mathcal U_{12n+4}$. See Remark~\ref{rk:reverse}. 

In \S\ref{sec:eisenstein}, we give the main result of the paper. Let
$K=\Q(\sqrt{-3})$. In 2011, G. Han and K. Ono connect $3$-core partitions
and ideals of the ring of Eisenstein integers~\cite[Theorem 1.4]{HanOno}
by proving an identity between modular forms. Ono has asked for a direct
explanation for this phenomenon. We will see that our construction allows
to describe $\mathcal E_n$ using this ring. In particular, we obtain in
Theorem~\ref{thm:eisenstein} an explicit bijection between the set of
$3$-core partitions and the ideals of the ring of Eisenstein integers of
norm $3n+1$. This explains the result of Han and Ono. 

In the last section, we give some consequence of our approach. First, we
connect the numbers of $3$-core partitions of $n$ and of $kn+\frac{k-1}3$
for any integer $k$ not divisible by $p$ and such that the $3$-valuation
at all prime with residue $2$ modulo $3$ is even. See
Formula~(\ref{eq:nombre}). This generalizes results of D. Hirschhorn,
and J.\,A. Sellers~\cite{HirschhornSellers3cores}, also proved by another
method by Baruah and Nath in~\cite{BaruahNath}. 
    Then we will prove that the new equality
\begin{equation}
\label{eq:amazing}
a_3(3n^2+(3^{k+1}+2)n+3^k)=a_3(n)a_3(n+3^k)
\end{equation}
holds for any integers $n\geq 1$ and $k\geq 0$. Furthermore, using the
crank $\mathfrak c$ of \S\ref{sec:part2}, we construct an explicit
bijection that explains it. This example illustrates
the advantage of the methods developed in the present work.
    Finally, we discuss a question of Han~\cite[5.2]{Han}.

\section{Construction of a crank} \label{sec:part2}

    For any symmetric matrix $A\in\operatorname{M}_2(\R)$, we denote its
corresponding orthogonal groups by $\mathcal O_2(A)$, that is
$$\mathcal O_2(A)=\{Q \in \operatorname{GL}_2(\R) \mid Q^T AQ=A\},$$
where $Q^T$ denotes the transposed matrix of $Q$. When $A=I$ is the
identity matrix, we simply write $\mathcal O_2(\R)$ for $\mathcal O_2(I)$.

    Let $k$ be a nonnegative integer. In this section, we consider the
sets 
$$\mathcal A_k=\{(x,y)\in\R^2 \mid x^2+y^2=k\}\quad\text{and}\quad
\mathcal B_k=\{(x,y)\in\R^2 \mid x^2+3y^2 = k\}.$$ 

    Note that we can interpret the sets $\mathcal A_k$ and $\mathcal B_k$
using matrices. For that, we write $D\in\operatorname{M}_2(\R)$ for the
diagonal matrix with entries $1$ and $3$. By abuse of notation, we
identify the elements of $\R^2$ with column vectors. Then we have $X\in
\mathcal A_k$ if and only if $X^TX=k$. Similarly, $Y\in\mathcal B_k$ if
and only if $Y^TDY=k$. Moreover, we remark that $\mathcal
O_2(\R)$ and $\mathcal O_2(D)$ act respectively on $\mathcal A_k$ and
$\mathcal B_k$ by left multiplication.

    Now, we denote by $P$ the diagonal matrix with entries $1$ and $\sqrt
3$. In particular, $P^T P=D$, and, by elementary bilinear algebra results,
conjugation and multiplication by $P^{-1}$ induce respectively a group
isomorphism between $\mathcal O_2(\R)$ and $\mathcal O_2(D)$ and a
bijection between $\mathcal A_k$ and $\mathcal B_k$.
We remark that $\mathcal U_k$ is the subset of $\mathcal B_k$ consisting
of points with integral coordinates.

\begin{theorem}
\label{thm:action}
Assume $k=12n+4$ for a nonnegative integer $n$. Then there is a cyclic
subgroup $G$ of $\mathcal O_2(D)$ of order $6$ that acts freely on
$\mathcal U_k$. Moreover, $G$ acts on the fibers of the map $$\mathfrak
c:\mathcal U_k\longrightarrow (\Z/3\Z)^2,\ (x,y)\longmapsto
(\overline{2x},\overline{2x+2y}),$$ where $\overline m$ denotes the
residue class of $m$ in $\Z/3\Z$.
\end{theorem}

\begin{proof}
    The orthogonal group is well-known in dimension $2$. It consists of
rotations and orthogonal symmetry. It follows that
\begin{align*}
\mathcal O_2(D)&=P^{-1}\mathcal O_2(\R)P\\
&=\left\{
\begin{pmatrix}
1&0\\
0&\frac 1 {\sqrt 3}
\end{pmatrix}
\begin{pmatrix}
\cos \theta& -\varepsilon \sin \theta\\ \sin\theta&\cos\theta
\end{pmatrix} \begin{pmatrix} 1&0\\ 0&\sqrt 3 \end{pmatrix} \mid \theta\in
\R,\,\varepsilon\in\{-1,1\}\right\}\\ &=\left\{ \begin{pmatrix}
\cos \theta& -\varepsilon \sqrt 3 \sin \theta\\
\frac 1 {\sqrt 3}\sin\theta&\cos\theta
\end{pmatrix}
\mid
\theta\in \R,\,\varepsilon\in\{-1,1\}\right\},
\end{align*}
where $P$ is defined as above. 
    Now, consider the element $R$ of $\mathcal O_2(D)$ associated
with the parameters $\theta=\pi/3$ and $\varepsilon=1$ in the above
description, that is
$$R=
\frac 1 2
\begin{pmatrix}
1&-3\\
1&1
\end{pmatrix}.$$
This is the conjugate by $P^{-1}$ of the rotation matrix $R_{\pi/3}$ of
angle $\pi/3$ of $\mathcal O_2(\R)$. In particular, $R$ has order $6$.

    We set $G = \langle R\rangle$. Then $G$ acts on $\mathcal B_k$
by left multiplication. Let $(x,y)\in\mathcal B_k$. Write $X=(x\ y)^T$.
Denote by $H$ the $G$-stabilizer of $(x,y)$. It is a subgroup of $G$ of
order $1$, $2$, $3$ or $6$. The order of $H$ is not $6$, because
$(x,y)\neq (0,0)$. We now exclude the orders $2$ and $3$. 
    Suppose $H$ has order $2$. Since $G$ is cyclic, it only has one
subgroup of order $2$, hence $H=\langle R^3\rangle$, and
$R^3 X=-X$. It follows that $X=-X$, that is $X=0$ which is
not possible. 
    Suppose $H$ has order $3$. Then $H=\langle R^2\rangle$. The
condition $R^2 X=X$ gives the system
$$\left\{
\begin{array}{rl}
-3x-3y&=0\\
x-3y&=0
\end{array}
\right..$$
Here again, $X=0$. It follows that $H$ is trivial and the $G$-action is
free. Each $G$-orbit then has size $6$. 

    As a subgroup of $\mathcal O_2(D)$, $G$ acts naturally on $\mathcal
B_k$ by left multiplication. We now check that this action stabilizes
$\mathcal U_k$. Let $(x,y)\in \mathcal U_k$. Then $x$ and $y$ are integers
and $x^2+3y^2=k$. Furthermore, since $k$ is even by assumption, we deduce
that $x^2 \equiv y^2\mod 2$, and $x$ and $y$ have the same parity. Thus,
the coordinates of 
\begin{equation}
\label{eq:imageR}
R
\begin{pmatrix}
x\\
y
\end{pmatrix}
=
\frac 1 2
\begin{pmatrix}
x-3y\\
x+y
\end{pmatrix}
\end{equation}
are integers, as required.

    Now, we study the residue modulo $3$ of the coordinates of a vector of
$\mathcal U_{k}$ under the action $R$. First, we observe that if $x=2x'$
is even, then $x'\equiv 2x\mod 3$, and if $x=2x'+1$ is an odd number, then
$x'\equiv 2(x-1)\mod 3$. 

    Let $X^T=(x,y)\in\mathcal U_k$. Write $(a,b)$ for the coordinates of
$RX$. As remarked above, $x$ and $y$ have the same parity. Assume first
that $x=2x'$ and $y=2y'$ are both even. Then~(\ref{eq:imageR}) gives
$$a=x'-3y'\equiv 2x\mod 3\quad\text{and}\quad
b=x'+y'\equiv  2(x+y)\mod 3.$$
Assume now that $x=2x'+1$ and $y=2y'+1$ are odd. Then
$$a=x'-3y'-1\equiv 2(x-1)-1\equiv 2x\mod 3$$
and
$$b=x'+y'+1\equiv 2(x-1) + 2(y-1)+1\equiv 2(x+y)\mod 3.$$
In any cases, we obtain
\begin{equation}
\label{eq:resteimage}
a\equiv 2x \mod 3\quad\text{and}\quad b\equiv 2(x+y)\mod 3,
\end{equation}
which does not depend on the parity of $x$ and $y$. Note that
$$\mathfrak c(x,y)=(\overline a,\overline b).$$

    Since $(x,y)\in \mathcal U_k$, we have $x^2\equiv 1 \mod 3$ because
$k\equiv 1 \mod 3$ by assumption. Hence, $x\equiv \pm 1\mod 3$ and there
is no condition on the residue modulo $3$ of $y$. The possible residue
classes modulo $3$ for the coordinates of $(x,y)$ are elements of
\begin{equation}
\label{eq:liste}
\{(\overline 1,\overline 1),\,(\overline 2,\overline 1),\,(\overline
1,\overline 0),\,(\overline 2,\overline 2),\,(\overline 1,\overline 2)\,
(\overline 2, \overline 0)\}.
\end{equation}
Now, using~(\ref{eq:resteimage}), we compute the residue modulo $3$ of the
coordinates of $RX$ for a vector $X$ whose coordinates has residue
in~(\ref{eq:liste}). We summarize the result in the following graph. The
vertices are labeled by the residues modulo $3$ of the coordinates of the
vector and an arrow between two vertices represents a multiplication by $R$.
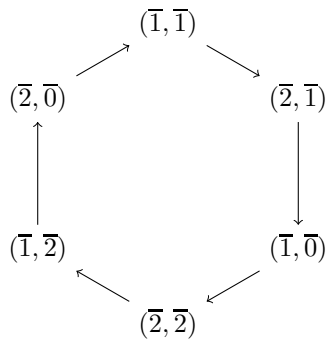
\label{graphe}
\begin{center}
\begin{tikzpicture}
\def\t{{"(\overline 1,\overline 1)","(\overline 2,\overline
1)","(\overline 1,\overline 0)","(\overline 2,\overline 2)","(\overline
1,\overline 2)","(\overline 2,\overline 0)"}}
\foreach \i [count=\ii from 0] in {90.0,
30.0,
-30.0,
-90.0,
-150.0,
-210.0}
\path (\i:20mm) node (p\ii) {$\pgfmathparse{\t[\ii]}\pgfmathresult$};
\foreach \i [count=\ii from 0] in {1,...,5,0}
\draw[->] (p\ii) -- (p\i);
\end{tikzpicture}
\end{center}
    We already remark that the $\Z/3\Z$-residue vector of $(x,y)$ lies
in~(\ref{eq:liste}). Then by the graph, the elements $\mathfrak
c(R^i(x,y))$ for $0\leq i\leq 5$ are exactly the ones of the
set~(\ref{eq:liste}). In particular,  the image of $\mathfrak c$ is the
set~(\ref{eq:liste}) and multiplication by $R$ permutes cyclically the
fibers of $\mathfrak c$.
\end{proof}

\begin{remark} Let $k=12n+4$ be a positive integer.
\label{rk:crank}
\begin{enumerate}[(i)]
\item In the proof the theorem, we only use the fact that $k$ is even and
has residue $1$ modulo $3$. By Chinese Remainder Theorem, this condition
is equivalent to $k\equiv 4\mod 6$. Suppose that $k=6q+4$ with $q$ an odd
integer. Then there is $m\in\Z$ such that $q=2m+1$, and $k=12m+10$. Assume
$(x,y)\in\mathcal U_k$. Then $(x,y)$ satisfies 
equation~(\ref{eq:equationmain}), and by reducing the equality modulo
$12$, we find that the equation $x^2+3y^2=10$ has a solution over
$\Z/12\Z$. However, an exhaustive computation in $\Z/12\Z$ shows that this
equation has no solution. Hence, $\mathcal U_{k} = \emptyset$, and we only
have to consider integers satisfying the assumption of the theorem.
\item Note that the group $G$ constructed in Theorem~\ref{thm:action}
does not lie in the orthogonal group over $\Z$ of $D$. This is a subgroup
of matrices over $\R$ that stabilizes the integral vectors of $\mathcal
B_k$.
\item Since $G$ permutes cyclically the fibers of $\mathfrak c$, each
fiber has the same cardinality. Hence, $\mathfrak c$ is a crank for
$\mathcal U_k$ for any $k$ a positive even integer. In particular, this
proves that
$$|U_k|\equiv 0\mod 6.$$
\item 
Each elements of~(\ref{eq:liste}) appear only once in all $G$-orbits. 
Hence, each $G$-orbits contain only one vector whose coordinates have
$\Z/3\Z$-residue $(\overline 1,\overline 1)$.
\end{enumerate}
\end{remark}

\section{Connection with $3$-core partitions}

    First, we roughly recall how to connect the set $\mathcal C$ of
characteristic vectors given in~(\ref{eq:enscar3}) and the Frobenius
symbol of $3$-core partitions. For more details, we refer
to~\cite[\S3.2]{BrNa2}. To any partition $\lambda$, we attach bijectively
two sets of nonnegative integers $\mathcal A_{\lambda}$ and $\mathcal
L_{\lambda}$ of the same size, respectively called the set of arms and of
legs of $\lambda$. Geometrically, these sets can be interpreted on the
Young diagram of $\lambda$ : an arm of $\lambda$ attached to a diagonal
box $d$ of the diagram is the number of boxes in the same row and to the
right of $d$. Similarly, the leg corresponding to $d$ is the one in the
same column and below $d$. This information is encoded into the Frobenius
symbol of $\lambda$, denoted by $\mathfrak F(\lambda) = (\mathcal
L_{\lambda} \mid \mathcal A_{\lambda})$.

    There is another geometric interpretation of the Frobenius symbol of
$\lambda$ in terms of a pointed $t$-abacus, where $t$ is a positive
integer. We recall it only for $t=3$. A pointed $3$-abacus is an abacus
with three runners labeled $0$, $1$ and $2$ equipped with a fence. On each
runner, slots both over and under the fence are labeled by the set of
nonnegative integers. In each slot, a white or black bead is drawn so that
the number of black beads over the fence is finite and is equal to the one
of white beads under the fence. The pointed $3$-abacus of $\lambda$ is
then obtained as follows. Let $0\leq i\leq 2$ and $q$ be a positive
integer. Then $3q+r\in\mathcal A_{\lambda}$ if and only if the pointed
$3$-abacus of $\lambda$ has a black bead over the fence at position $q$ on
the runner $i$. Similarly, it has a white bead under the fence at position
$q$ on the runner $2-r$ if and only if $3q+r$ is a leg of $\lambda$.

    Let $\lambda$ be a $3$-core partition, and $\mathcal T$ its pointed
$3$-abacus. For $0\leq i\leq 2$, if the $i$th-runner of $\mathcal T$ has
no white bead under the fence, we write $c_i$ for the number of black
beads over the fence. Otherwise, we write $-c_i$ for the number of white
beads under the fence of the $i$th-runner of $\mathcal T$. Then we set
$c_{\lambda} = (c_0,c_1,c_2)$. This tuple is called the characteristic
vector of $\lambda$. In~\cite[\S3.3]{BrNa2}, it is proved that the map 
$$\varphi:\mathcal E\longrightarrow \mathcal C,\ \lambda\longmapsto c_{\lambda}$$
is a well-defined bijection. In the following, we identify $\Z^2$ with
$\mathcal C$ using the map $\Z^2\rightarrow \mathcal C,\ (x,y)\mapsto
(-x-y,x,y)$. By this abuse of notation, we also will write $c_{\lambda} =
(x,y)$ for the characteristic vector of $\lambda$. 

    For positive integer $n$, we denote by $\mathcal C_n$ for the elements
of $\mathcal C$ corresponding to $3$-core partitions of size $n$.

\begin{lemma}
\label{lemma:injection}
Let $n$ be a positive integer. Then the map
$$\vartheta:\mathcal C_n\longrightarrow \mathcal U_{12n+4},\
(x,y)\longmapsto (6x+3y+1,3y+1)$$
is injective.
\end{lemma}

\begin{proof}
    For any $3$-core partition $\lambda$ with characteristic vector
$c_{\lambda} = (x,y)\in\Z^2$, recall that gives $|\lambda| =
3x^2+3xy+3y^2+x+2y$. See for example~\cite[Corollary 3.20]{BrNa2}.
However,
$$3x^2+3xy+3y^2+x+2y=\frac{1}{12}(6x+3y+1)^2+\frac{1}4(3y+1)^2+\frac 1
3.$$
Furthermore, $\lambda\in\mathcal E_n$ if and only if $|\lambda|=n$ if and
only if
$$(6x+3y+1)^2+3(3y+1)^2=12n+4.$$
In particular, $(6x+3y+1,3y+1)\in\mathcal U_{12n+4}$, and the map
$\vartheta$ is well-defined. It is clearly injective.
\end{proof}

    The next result connects $\mathcal E_n$ and $\mathcal U_{12n+4}$. This
gives a way to parameterize the solutions of
Equation~(\ref{eq:equationmain}).

\begin{theorem}
\label{thm:main}
Let $n$ be a nonnegative integer. Then
\begin{align*}
\mathcal
U_{12n+4} = & \{(6x+3y+1,3y+1),\, (3x-3y-1,3x+3y+1),\, (-3x-6y-2,3x),\,\\
            &(-6x-3y-1,-3y-1),\,(-3x+3y+1,-3x-3y-1),\\
            &(3x+6y+2,-3x) \mid (x,y)\in\mathcal C_n\}.
\end{align*}
\end{theorem}

\begin{proof}
    Write $R$ for the matrix of $\mathcal O_2(D)$ described in the proof
of Theorem~\ref{thm:action}. Since $12n+4$ is even, the group
$G=\langle R \rangle$ acts on $\mathcal U_{12n+4}$, and each $G$-orbits
has size $6$. Let $\vartheta:\mathcal C_n\rightarrow \mathcal U_{12n+4}$
be the injective map given in Lemma~\ref{lemma:injection}. First, we will
prove that $\operatorname{Im}(\vartheta)$ is the set 
$$\mathcal F_n=\{
(u,v)\in\mathcal U_{12n+4} \mid u\equiv 1\mod 3\ \text{ and }\ v\equiv
1\mod 3\}.$$ 
    It is clear that $\operatorname{Im}(\vartheta)\subseteq \mathcal F_n$
by definition of $\vartheta$. Conversely, let $(u,v)\in\mathcal F_n$.
Write $z\in \Z$ such that $u=z+v$. We have
$$12n+4=u^2+3v^2=(z+v)^2+3v^2=z^2+2zv+4v^2.$$
    Considering this equality modulo $2$, we obtain that $z^2\equiv 0\mod
2$, hence $z$ is even. On the other hand, $z+v$ and $v$ have the same
residue modulo $3$ since $(z+v,v)\in\mathcal F_n$. This implies that
$z\equiv 0\mod 3$. Hence, $z$ is divisible by $6$, and there is $x\in\Z$
such that $z=6x$. Furthermore, $v\equiv 1\mod 3$. Thus, there is $y\in\Z$
such that $v=3y+1$, and
$$\vartheta(x,y)=(6x+3y+1,3y+1)=(z+v,v)=(u,v),$$
proving that $\mathcal F_n\subseteq \operatorname{Im}(\vartheta)$. Hence,
$$\operatorname{Im}(\vartheta)=\mathcal F_n.$$

    However, we proved in Theorem~\ref{thm:action} that $\mathcal F_n$ is
a set of representatives of the $G$-orbits on $\mathcal U_{12n+4}$. Note
that for any $X^T=(u,v)\in\mathcal U_{12n+4}$, 
\begin{eqnarray}
RX=
\frac 1 2
\begin{pmatrix}
u-3v\\
u+v
\end{pmatrix},&
R^2X=
\dfrac 1 2
\begin{pmatrix}
-u-3v\\
u-v
\end{pmatrix},&
R^3X=
-X,\\
R^4X=
\frac 1 2
\begin{pmatrix}
-u+3v\\
-u-v
\end{pmatrix},&
R^5X=
\dfrac 1 2
\begin{pmatrix}
u+3v\\
-u+v
\end{pmatrix},\nonumber
\end{eqnarray}
Finally, we conclude by computing the orbit of $\vartheta(x,y)$ for
$(x,y)\in\mathcal C_n$ using these relations.
\end{proof}

\begin{remark}
For any $(u,v)\in\mathcal U_{12n+4}$, we write $G\cdot(u,v)$ for
its $G$-orbit. By Theorem~\ref{thm:main}, we have
$$u(12n+4)=|\mathcal U_{12n+4}|=\sum_{(x,y)\in \mathcal C_n} |G\cdot
\vartheta(x,y)|=6|\mathcal C_n|= 6|\mathcal E_n|.$$
In particular, this proves~\cite[Theorem 3.1]{BaruahNath}.
\end{remark}



\begin{remark}
\label{rk:reverse}
By Theorem~\ref{thm:action} and Theorem~\ref{thm:main} the fiber of
$(\overline 1,\overline 1)$ with respect to the crank $\mathfrak c$ is the
set $\mathcal C_n$. Then we obtain a way to recover $\mathcal E_n$ from
$\mathcal U_{12n+4}$.
\end{remark}

\begin{example}
We consider the equation $x^2+3y^2=448$ over $\Z^2$. First, we remark that
$448=12\cdot 37+4$. Then we have to describe $\mathcal E_{37}$. On the
other hand, $37=4\times 9+1$. By~\cite[Corollary 5.3]{BrNa2}, $\mathcal
E_{9}$ and $\mathcal E_{37}$ are in bijection. Now, we remark by looking
at its character table that the symmetric group $\Sym_9$ has exactly two
defect zero characters labeled by the partitions
$$\lambda=(5,3,1) \quad\text{and}\quad\mu=(3,2^2,1^2).$$
These are the $3$-core partitions of $9$. Furthermore, we have $\mathfrak
F(\lambda)=(2,0\mid 1,4)$ and $\mathfrak F(\mu)=\mathfrak
F(\lambda^*)=(4,1|0,2)$, and the corresponding pointed $3$-abacus are
\medskip

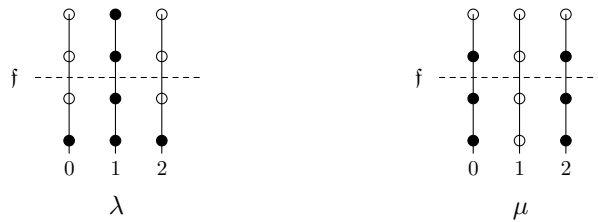
\begin{center}
\begin{tabular}{lll}
\begin{tikzpicture}[line cap=round,line join=round,>=triangle
45,x=0.7cm,y=0.7cm, scale=0.8,every node/.style={scale=0.8}]

\draw [dash pattern=on 2pt off 2pt](-2.5,1.5)-- (1.5,1.5);
\draw(-3,1.5)node{$\mathfrak f$};

\draw (-2,-0.3) node[anchor=north west] {$0$};

\draw (-1.7,3)-- (-1.7,-0.3);

\begin{scriptsize}

\draw (-1.7,2) circle (2.5pt);

\draw (-1.7,3) circle (2.5pt);

\draw (-1.7,1) circle (2.5pt);

\draw (-1.7,0)[fill=black] circle (2.5pt);
\end{scriptsize}

\draw (-0.9,-0.3) node[anchor=north west] {$1$};

\draw (-0.6,3)-- (-0.6,-0.3);

\begin{scriptsize}
	
\draw  (-0.6,2)[fill=black] circle (2.5pt);
\draw (-0.6,3)[fill=black]  circle (2.5pt);

\draw (-0.6,1)[fill=black] circle (2.5pt);

\draw (-0.6,0)[fill=black]  circle (2.5pt);
\end{scriptsize}

\draw (0.2,-0.3) node[anchor=north west] {$2$};

\draw (0.5,3)-- (0.5,-0.3);

\begin{scriptsize}

\draw  (0.5,2) circle (2.5pt);
\draw (0.5,3) circle (2.5pt);
\draw (0.5,1) circle (2.5pt);
\draw (0.5,0)[fill=black]  circle (2.5pt);
\end{scriptsize}

\end{tikzpicture}
&
\hspace{2cm}
&
\begin{tikzpicture}[line cap=round,line join=round,>=triangle
45,x=0.7cm,y=0.7cm, scale=0.8,every node/.style={scale=0.8}]

\draw [dash pattern=on 2pt off 2pt](-2.5,1.5)-- (1.5,1.5);
\draw(-3,1.5)node{$\mathfrak f$};

\draw (-2,-0.3) node[anchor=north west] {$0$};

\draw (-1.7,3)-- (-1.7,-0.3);

\begin{scriptsize}

\draw (-1.7,2)[fill=black] circle (2.5pt);

\draw (-1.7,3) circle (2.5pt);

\draw (-1.7,1)[fill=black] circle (2.5pt);

\draw (-1.7,0)[fill=black] circle (2.5pt);
\end{scriptsize}

\draw (-0.9,-0.3) node[anchor=north west] {$1$};

\draw (-0.6,3)-- (-0.6,-0.3);

\begin{scriptsize}
	
\draw  (-0.6,2) circle (2.5pt);
\draw (-0.6,3)  circle (2.5pt);

\draw (-0.6,1) circle (2.5pt);
\draw (-0.6,0)  circle (2.5pt);

\end{scriptsize}

\draw (0.2,-0.3) node[anchor=north west] {$2$};

\draw (0.5,3)-- (0.5,-0.3);

\begin{scriptsize}

\draw  (0.5,2)[fill=black] circle (2.5pt);
\draw (0.5,3) circle (2.5pt);
\draw (0.5,1)[fill=black] circle (2.5pt);
\draw (0.5,0)[fill=black] circle (2.5pt);

\end{scriptsize}
\end{tikzpicture}\\
\hspace{1.4cm}$\lambda$&&\hspace{1.4cm}$\mu$
\end{tabular}
\end{center}
\medskip
The characteristic vectors of $\lambda$ and $\mu$ are then $(-1,2,-1)$ and
$(1,-2,1)$. Moreover, using the bijection~\cite[Corollary 5.3]{BrNa2}, we
deduce that the $3$-core partitions of $37$ have characteristic vectors
$(3,-4,1)$ and $(1,4,-3)$, and
$$\vartheta(-4,1)=(-20,4)\quad\text{and}\quad
\vartheta(4,-3)=(16,-8).$$
We conclude with Theorem~\ref{thm:main} that 
\begin{eqnarray*}
\mathcal U_{448}=&\{( -20, 4 ), (-16, -8 ), ( 4, -12 ), ( 20, -4 ), ( 16,
8 ), ( -4, 12 ),\\ 
&( 16, -8 ), ( 20, 4 ), ( 4, 12 ), ( -16, 8 ), ( -20, -4 ), ( -4, -12 )
\}.
\end{eqnarray*}

\begin{remark}
Consider the map 
$$\operatorname{conj}:\mathcal
U_{12n+4}\longrightarrow \mathcal U_{12n+4},\,(u,v)\longmapsto
\left(\frac{1}2(-u+3v),\frac 1 2 (u+v)\right).$$
This map is well-defined, because for $(u,v)\in\mathcal U_{12n+4}$, we
have $\frac{1}2(-u+3v)\in\Z$ and $\frac 1 2 (u+v)\in\Z$ since $u\equiv
v\mod 2$, and
$$\left(\frac{1}2(-u+3v)\right)^2 + \left(\frac 1 2
(u+v)\right)^2=u^2+3v^2=12n+4.$$
    Now, for a partition $\lambda$ of $n$, we write $\lambda^*$ for its
conjugate partition. By~\cite[Proposition 3.12]{BrNa2}, if
$c_{\lambda}=(x,y)$, then $c_{\lambda^*}=(-x,x+y)$. Furthermore, for any
$3$-core partition $\lambda$, we have $$\operatorname{conj}\circ
\vartheta(c_{\lambda})=\vartheta(c_{\lambda^*}).$$ Hence, the image under
$\vartheta$ of the set of self-conjugate $3$-cores of $n$ consisting of
the vectors $(u,v)\in\mathcal U_{12n+4}$  such that $(u,v)$ lies in the
$1$-eigenspace of 
$$S=
\frac 1 2
\begin{pmatrix}
-1&3\\
1&1
\end{pmatrix}.$$
This is the set of $(u,u)\in\Z^2$ such that $u\equiv 1 \mod 3$ and
$u^2=3n+1$. In particular, there is at most one self-conjugate $3$-core
partition of $n$. Let $(u,u)$ be such an element and $k$ be any integer
not divisible by $3$. Let $\varepsilon\in\{-1,1\}$ be such that
$\varepsilon k\equiv 1\mod 3$.
Then $\varepsilon ku\equiv 1\mod 3$ and
$$(\varepsilon k u)^2=3\left(k^2n+\frac{k^2-1}3\right)+1.$$
It follows that the map $\vartheta^{-1}(u,u) \mapsto
\vartheta^{-1}(\varepsilon ku,\varepsilon ku)$ 
from the set of self-conjugate $3$-core partitions of $n$ into the ones of
$3$-core partitions of $k^2n+\frac{k^2-1}3$ is well-defined and bijective.
Hence,
$$\operatorname{asc}_3(n)=\operatorname{asc}_3\left(k^2n+\frac{k^2-1}3\right),$$
where $\operatorname{asc}_3(n)$ denotes the number of self-conjugate
$3$-core partitions of $n$. This generalizes~\cite[Theorem 3.6]{BaruahNath}.

    Note that the group generated by $R$ and $S$ is a dihedral group of
order $12$.
\end{remark}
\end{example}

\section{Connection with the ring of Eisenstein integers}
\label{sec:eisenstein}

    Theorem~\ref{thm:main} shows that we can recover $\mathcal E_n$ from
$\mathcal U_{12n+4}$. See also Remark~\ref{rk:reverse}. In this section,
we give another way to determine $\mathcal U_{12n+4}$ using the ring of
Eisenstein integers $\mathcal R=\Z[\mathbf j]$, where $\mathbf
j=e^{2i\pi/3}$. We write $\mathcal N:\mathcal R\rightarrow \N,\,z\mapsto
z\overline z$ for the norm of $\mathcal R$, where $\overline z$ denotes
the complex conjugation. Many results here are standard, but we recall
them for the convenience of the reader. Let $u,\,v\in\Z$. We have $\frac u
2 + \frac {v\sqrt 3} 2 i = \frac {u+v} 2 + v \mathbf j$, thus $\frac u 2 +
\frac {v\sqrt 3} 2 i \in \mathcal R$ if and only if $u$ and $v$ have the
same parity. Note also that, in this case, we have
$$\mathcal N\left(\frac u 2 + \frac{v\sqrt 3} 2 \,i\right)=\frac 1
4(u^2+3v^2).$$ 
    However, as remarked in the proof of Theorem~\ref{thm:action}, each
element of $\mathcal U_{12n+4}$ has this property, hence the map
\begin{equation}
\label{eq:defrho}
\rho:\mathcal U_{12n+4}\longrightarrow \mathcal R,\, (x,y)\longmapsto
\frac x 2+\frac {y\sqrt 3} 2 i
\end{equation}
is well-defined. Furthermore,
$(u,v)\in \mathcal U_{12n+4}$ if and only if
$$\mathcal N(\rho(u,v))=3n+1.$$
Write $\Pi_1$ and $\Pi_2$ for the sets of prime numbers with residue $1$
and $2$ modulo $3$, respectively. For any integer $k$, we denote by
$\nu_p(k)$ the $p$-adic valuation of $k$.
    By~\cite[Proposition 9.1.4]{IrelandRosen}, each $p\in \Pi_2$ is
irreducible in $\mathcal R$, and for $p\in \Pi_1$, there is an irreducible
element $x_p \in \mathcal R$ such that $p=\mathcal N(x_p)$. Furthermore,
we define
\begin{equation}
\label{eq:defVn}
\mathcal V_n=\prod_{p \in \Pi_1} \{0 \leq j_p \leq \nu_p(3n+1) \},
\end{equation}
and for any $\underline j=(j_p \mid p \in \Pi_1)\in\mathcal V_n$, we set
\begin{equation}
\label{eq:defxj}
x_{\underline j}= \prod_{p \in \Pi_1} x_p^{j_p}\
\overline{x_p}^{\nu_p(3n+1) - j_p}. 
\end{equation}

\begin{theorem}
\label{thm:eisenstein}
The set $\mathcal U_{12n+4}$ is nonempty if and only if $\nu_p(3n+1)$ is
even for all $p\in \Pi_2$. In this case, we set 
\begin{equation}
\label{eq:defqnyj}
q_n = \prod_{p \in \Pi_2} p^{\frac 1 2 \nu_p(3n+1)} \quad \text{and}
\quad y_{\underline j} = x_{\underline j} q_n \quad \text{for }
\underline j \in \mathcal V_n.
\end{equation}
Then, for any $\underline j \in \mathcal V_n$,
there are integers $u_{\underline j}$ and $v_{\underline j}$ with residue
$1$ modulo $3$ such that 
$$y_{\underline j} = \frac 1 2 \left(u_{\underline j} + v_{\underline
j}\, i\sqrt 3 \right),$$ 
and
$$\mathcal C_n=\{\vartheta^{-1}(u_{\underline j},\,v_{\underline j}) \mid \underline
j \in \mathcal V_n\},$$
where $\vartheta:\mathcal C_n\rightarrow \mathcal U_{12n+4}$ is the map
defined in Lemma~\ref{lemma:injection}.
\end{theorem}

\begin{proof}
First, we have
\begin{align*}
3n+1 &= \prod_{p\in\Pi_1}p^{\nu_p(3n+1)} \cdot \prod_{p\in \Pi_2}
p^{\nu_p(3n+1)}\\
&= \prod_{p\in \Pi_1} x_p^{\nu_p(3n+1)}\, \overline{x}_p^{\nu_p(3n+1)}
\cdot \prod_{p\in \Pi_2}p^{\nu_p(3n+1)},
\end{align*}
that gives a factorization into irreducible elements in $\mathcal R$. Now,
we note that there is $z=\frac a 2 + \frac b 2 i\sqrt 3 \in \mathcal R$
such that $3n+1=\mathcal N(z) = z\overline z$ if and only if
$a^2+3b^2=12n+4$ if and only if $(a,b)\in\mathcal U_{12n+4}$. 
    Assume there is $z\in \mathcal R$ such that $3n+1=\mathcal N(z)$.
By~\cite[Proposition 1.4.2]{IrelandRosen}, $\mathcal R$ is a Euclidean
domain. Hence, for any $p\in \Pi_2$ such that $\nu_p(3n+1)\neq 0$, $p$
divides $z$ or $\overline z$ in $\mathcal R$. We can assume without loss
of generality that $p$ divides $z$. Then $p=\overline p$ divides
$\overline z$, proving that $p^2$ divides $3n+1$ in $\Z$. It follows that
$\nu_p(3n+1)$ is even. Furthermore, by the uniqueness (up to a unit
element) of the decomposition into irreducible elements in a Euclidean
ring, we deduce first that $q_n$ divides $z$, and then that there are
$\xi\in \mathcal R^{\times}$ and integers $0 \leq j_p \leq \nu_p(3n+1)$
and $0 \leq j_p' \leq \nu_p(3n+1)$ for any $p\in\Pi_1$, such that
$$z=\xi q_n \prod_{p\in\Pi_1} x_p^{j_p} \, {\overline{x_p}}^{j_p'}.$$
Hence, 
$$3n+1=z\overline z=q_n^2\prod_{p\in\Pi_1} x_p^{j_p+j'_p}\
{\overline{x_p}}^{j_p + j_p'}.$$
By the uniqueness of the decomposition, we deduce that $j_p + j_p' =
\nu_p(3n+1)$. It follows that 
\begin{equation}
\label{eq:factorisation}
z=\xi q_n \prod_{p\in\Pi_1} x_p^{j_p}\,
\overline{x_p}^{\nu_p(3n+1)-j_p} = \xi q_n x_{\underline j}=\xi
y_{\underline j},
\end{equation}
with $\underline{j} = (j_p\mid p\in\Pi_1)$.
    Conversely, any factorization in $\mathcal R$ of $3n+1$ of the form $z
\overline z$ is as in~(\ref{eq:factorisation}) for some $\xi\in
\mathcal R^{\times}$ and $\underline j \in \mathcal V_n$.

    Now, by~\cite[Proposition 9.1.1]{IrelandRosen}, $\mathcal R^{\times}$
has order $6$ and $\mathcal R^{\times}=\langle \mathbf j+1 \rangle$. Note
that $\mathcal R^{\times}$ acts on $\rho(\mathcal U_{12n+4})$ by
multiplication since the elements of $\mathcal R^{\times}$ have norm $1$.
On the other hand, if $z = \frac 1 2(a+bi\sqrt 3)$ for $a,\,b\in \Z$, then
$$(\mathbf j+1)z=\frac 1 2\left( \frac 1 2(a-3b)+\frac 1 2 (a+b) i \right).$$
Comparing with~(\ref{eq:imageR}), we see that $\langle R \rangle$ acts on
$\mathcal C_n$ like $\mathcal R^{\times}$ on $\rho(\mathcal U_{12n+4})$. In
particular, by Theorem~\ref{thm:main}, for any $\underline j \in \mathcal
V_n$, there is $\xi\in \mathcal R^{\times}$ such that $\xi q_n x_{\underline
j}=\frac 1 2(u_{\underline j}+v_{\underline j} i)$ satisfies
$u_{\underline j}\equiv v_{\underline j}\equiv 1 \mod 3$. Then
$\vartheta(\mathcal C_n)=\{(u_{\underline j}, v_{\underline j}) \mid
\underline j\in\mathcal V_n\}$ by Theorem~\ref{thm:main}, as required.
\end{proof}

\begin{remark}
\label{rk:eisenstein}
By Theorem~\ref{thm:main} and Theorem~\ref{thm:eisenstein}, the
description of the set of $3$-core partitions of $n$ is reduced to 
\begin{enumerate}[(i)]
\item The determination of the prime decomposition in $\Z$ of $3n+1$.
\item The decomposition of the prime $p\in \Pi_1$ dividing $3n+1$ into
irreducible elements in $\Z[\mathbf j]$. This factorization is of the form
$z\overline z$ for irreducible elements $z$ and $\overline z$ of
$\Z[\mathbf j]$.
\end{enumerate}
\end{remark}

\begin{remark}
\label{rk:t3}
By Theorem~\ref{thm:eisenstein}, we have
$$|\mathcal E_n|=|\mathcal C_n|=|\{(u_{\underline j},\,v_{\underline
j})\mid \underline j \in \mathcal V_n \}|=|\mathcal V_n|=\prod_{p\in
\Pi_1} (\nu_p(3n+1)+1),$$
and we recover a well-known result first proven by Granville and
Ono~\cite{GranvilleOno}, and also fund in~\cite[Theorem
6]{HirschhornSellers3cores} and~\cite[Corollary 3.3]{BaruahNath}.
\end{remark}

    Theorem~\ref{thm:main} can also be useful to find the set of $3$-core
partitions of an integer $n$; a hard problem in general.

\begin{example}
We determine the $3$-core partitions of $n=100$. First, consider the
equation $x^2+3y^2=4\cdot 301$. Then describe $\mathcal U_{1204}$ using
Theorem~\ref{thm:eisenstein}. We remark that $301=7\times 43$, and
$7$ and $43$ lie in $\Pi_1$. However, $7=\mathcal N(2+ i\sqrt 3)$ and
$43=\mathcal N(4+3i\sqrt 3)$. Set $x_3=2+i\sqrt 3$ and $x_{43}=4+3i\sqrt
3$, and consider
$$
\alpha_1=\overline{x_7 x_{43}}=\frac{1}{2}(-2-20i\sqrt 3),\
\alpha_2=x_7 \overline{x_{43}}=\frac{1}{2}(34-4i\sqrt 3),$$
$$
\alpha_3=\overline{x_7} x_{43}=\frac{1}{2}(34+4 i\sqrt 3)\quad \text{and}
\quad
\alpha_4=x_7 x_{43}=\frac{1}{2}(-2+20i\sqrt 3),\
$$
Thus, $y_{0,0}=\alpha_1$ and $y_{0,1}=\alpha_3$, and
$$y_{1,0}=(1+\mathbf j)^2\alpha_2=\frac{1}{2}(-11+19i\sqrt
3)\quad\text{and}\quad 
y_{1,1}=(1+\mathbf j)^2\alpha_4=\frac{1}{2}(-29-11i\sqrt
3).
$$
Thus,
$$\mathcal F_{100}=\{(-2,-20),(-11,19),(-29,-11),(34,4)\}.$$
Now, as in the proof of Theorem~\ref{thm:main}, we obtain
$$\vartheta(3,-7)=(-2,-20),\quad
\vartheta(-5,6)=(-11,19),\quad
\vartheta(-3,-4)=(-29,-11)$$
and
$$
\vartheta(5,1)=(34,4).
$$
Hence, $$\mathcal
C_{100}=\{(4,3,-7),\,(-1,-5,6),\,(7,-3,-4),\,(-6,5,1)\}.$$
Let $\lambda$ be the $3$-core partition with characteristic vector
$(4,3,-7)$. Using~\cite[Lemma 3.19]{BrNa2}, we obtain 
$$\mathfrak F(\lambda)=(18,15,12,9,6,3,0\mid 0,1,3,4,6,7,9),$$
and we deduce that $$\lambda=(10,9^2,8^2,7^2,6^2,5^2,4^2,3^2,2^2,1^2).$$

    Similarly, we find that the $3$-core partitions of $100$ with
characteristic vectors $(-1,-5,6)$, $(7,-3,-4)$ and $(-6,5,1)$ are
respectively 
$$(18,16,14,12,10,8,6,4,3^2,2^2,1^2),\quad
(19, 17, 15, 13, 11, 9, 7, 5, 3, 1 )\quad \text{and}$$
$$(14, 12, 10, 8, 7^2, 6^2, 5^2, 4^2, 3^2, 2^2, 1^2).$$
\end{example}

\section{Applications}
\label{sec:applications}

\subsection{Generalization of Hirschhorn-Sellers formula}
Let $n$ be a positive integer. Let $k$ be a positive integer not divisible
by $3$. Write $m$ be for the product of prime factors of $k$ with residue
$2$ modulo $3$. Assume $m$ is a square. Then 

\begin{equation}
\label{eq:nombre}
a_3\left(kn+\frac{k-1}3\right)=a_3(n)\prod_{p\in\Pi_1}\frac{\nu_{p}(k)+\nu_p(3n+1)
+ 1} {\nu_p(3n+1)+1},
\end{equation}
where $a_3(n)=|\mathcal E_n|$.

Indeed, we first observe that 
$m\equiv 1\mod 3$ since $m$ is a square, and it follows that
$k\equiv 1\mod 3$. Set $N=kn+\frac 1 3(k-1)$. We have
$$
3N+1=k(3n+1)=\prod_{p\in \Pi_1}p^{\nu_p(k)+\nu_p(3n+1)}m'm,
$$
where $m'$ is the product (with multiplicity) of the prime factors
with residue $2$ modulo $3$ of $3n+1$. By assumption, $mm'$ is a square if
and only if $m'$ is.
Hence, Theorem~\ref{thm:eisenstein} implies that $a_3(n)=0$ if
and only if $a_3(N)=0$. Assume $a_3(n)\neq 0$. Now, again 
by Theorem~\ref{thm:eisenstein}, we obtain
$$a_3(N)=\prod_{p\in\Pi_1}(\nu_p(k)+\nu_p(3n+1)+1)
=\prod_{p\in\Pi_1}\frac{\nu_p(k)+\nu_p(3n+1)+1}{\nu_p(3n+1)+1}a_3(n),$$
as required.

\begin{remark}
\noindent
\begin{enumerate}[(i)]
\item When $k$ is a square not divisible by $3$ and with no prime factors with
residue $1$ modulo $3$, then
$$a_3\left(kn+\frac{k-1}3\right)=a_3(n).$$
This generalizes~\cite[Corollary 8]{HirschhornSellers3cores}
and~\cite[Theorem 5]{BaruahNath}.
\item Assume $k$ is not divisible by $3$ and has no prime factors
congruent to $1$ modulo $3$. In~\cite[Theorem 5.1]{BrNa2}, we prove the
injectivity of the map 
$$\mathcal C_n\longrightarrow \mathcal C_{k^2n+\frac 1 3(k^2-1)},\
(x,y)\mapsto (\varepsilon kx,\varepsilon(ky+q)),$$
where $\varepsilon\in\{-1,1\}$ and $q\in \N$ are such that
$k=3q+\varepsilon$. By (i), this map is bijective.
\end{enumerate}
\end{remark}

\subsection{An amazing equality} 

    In this part, we derive~(\ref{eq:amazing}) from
Theorem~\ref{thm:eisenstein}. Then using the crank $\mathfrak c$ of
\S\ref{sec:part2} and the map $\vartheta$ of Lemma~\ref{lemma:injection},
we construct an explicit bijection that explains this equality. 

    Let $n$ and $m$ be two positive integers such that $3n+1$ and $3m+1$
are coprime. Then the prime decomposition of $(3n+1)(3m+1)$ is the
``concatenation'' of the one of $3n+1$ with the one of $3m+1$.  Since
$(3n+1)(3m+1)=3(3mn+m+n)+1$, Theorem~\ref{thm:eisenstein} gives
$$a_3(3mn+m+n)=a_3(n)a_3(m).$$
Assume $m=n+3^k$ for some $k\in\N$. If $d$ divides $3n+1$ and $3m+1$, then
$3n+1\equiv 0\equiv 3m+1\mod d$. Hence, $n\equiv m \mod d$ since $d$ and
$3$ are coprime, and $3^k\equiv 0\mod d$. Using again that $d$ and $3$ are
coprime, we deduce that $d=1$. The integers $3n+1$ and $3m+1$ are then
coprime, and Equality~\ref{eq:amazing} holds.

    Now, we assume $a_3(n)$ and $a_3(m)$ are non-zero. For any positive
integer $t$, and any $\underline j\in\mathcal V_t$, we write
$x_{\underline j}^{(t)}$ and $y_{\underline j}^{(t)}$ for the elements
defined in~(\ref{eq:defxj}) and in~(\ref{eq:defqnyj}). Set
$N=3n^2+(3^{k+1}+2)n+3^k$. Since $3n+1$ and $3m+1$ are coprime, we obtain
$$\mathcal V_N = \mathcal V_n \times \mathcal V_m,\quad
x_{\underline j}^{(N)}=x_{\underline j_1}^{(n)}y_{\underline
j_2}^{(m)}\quad\text{and}\quad q_N=q_nq_m,$$
where $\underline j=(\underline j_1,\underline j_2)$. Hence $y_{\underline
j}^{(N)}=y_{\underline j_1}^{(n)} y_{\underline j_2}^{(m)}$, and
any factorizations of $3N+1=z\overline z$ satisfies $z=z_nz_m$ 
$z_n$ and $z_m$ are equal up to a unit of $\Z[\mathbf j]$ to $y_{\underline
j_1}^{(n)}$ and $y_{\underline j_2}^{(m)}$, respectively. However, by
Theorem~\ref{thm:eisenstein}, these elements are equal (up to a unit
element) to $\rho\circ \vartheta (x,y)$ and
$\rho\circ\vartheta (x',y')$ for some $(x,y)\in\mathcal C_n$ and
$(x',y')\in \mathcal C_m$, where $\rho$ is the map defined
in~(\ref{eq:defrho}). Furthermore,
$$\rho\circ \vartheta (x,y)\rho\circ\vartheta (x',y')=(X,Y),$$
where $X=\frac{1}{2}(-1+3x-3y+3x'-3y'+18xx'+9xy'+9yx'-9yy')$ and 
$Y=\frac{1}2(9yy'+1+9yx'+3y'+3x+9xy'+3x'+3y)$. We remark that 
$$\mathfrak c(2X,2Y)=(\overline 2,\overline 1).$$
By the graph page~\pageref{graphe}, we deduce that the element of
$\mathcal U_{12N+4}$ which lies in the image of $\vartheta$ and on the
$\langle
R\rangle$-orbit of $(2X,2Y)$ is  $R^{-1}(2X,2Y) =(6a+3b+1,3b+1)$,
where 
$$ a = x + x' + 3 x x' + 3 x y' + 3 y x'\quad\text{and}\quad
b = y + y' - 3 x x' + 3 y y'.$$
This proves that the map $\alpha:\mathcal C_n\times \mathcal
C_{n+3^k}\longrightarrow \mathcal C_{3n^2+(3^{k+1}+2)n+3^k)}$ is defined for
any $(x,y)\in\mathcal C_n$ and any $(x',y')\in\mathcal C_{n+3^k}$ by
$$\alpha((x,y),(x',y'))=
(x+x'+3xx'+3xy'+3yx',y+y'-3xx'+3yy')$$
is a bijection.

\subsection{Remarks on a question of Han}
    In~\cite[5.2]{Han}, Han conjectured a criteria characterizing integers
with no $3$-cores. We now discuss this question.

\begin{enumerate}[(i)]
\item First, we remark that if $N=4^mn+\frac 1 3 (10\cdot 4^{m-1}-1)$ for
some integers $n\geq 0$ and $m\geq 1$, then 
$$3N+1=4^{m-1}(12n+10)=2^{2m-1}(6n+5).$$
However, $6n+5$ is odd, hence $\nu_2(3N+1)=2m-1$ is odd. Since
$2\in\Pi_2$, the first part of Theorem~\ref{thm:eisenstein} gives
$a_3(N)=0$. This proves the point (i) of~\cite[5.2]{Han}.
\item Now, we focus on the point (ii) of~\cite[5.2]{Han}, asserting
that if there are integers $n\geq 0$, $m\geq 1$, $k\geq 1$ with
$m\not\equiv 2k-1 \mod (6k-1)$ such that 
\begin{equation}
\label{eq:defN}
N=(6k-1)^2n+(6k-1)m+4k-1,
\end{equation}
then $a_3(N)=0$. Consider the case $N=58$. We have $a_3(58)=2$. We can see
that by noting that $3N+1=5^2\cdot 7$ and by applying
Theorem~\ref{thm:eisenstein}. However,
$$58=(6k-1)m+4k-1$$ 
for $m=1$ and $k=6$. Since $m\not\equiv 2k-1\mod 6k-1$, $58$ satisfies the
previous assumptions, but $a_3(58)\neq 0$, giving a counterexample of the
statement.

Note that
\begin{align*}
3N+1&=3(6k-1)^2n+3(6k-1)m+12k-2\\
&=(6k-1)(3n(6k-1)+3m+2).
\end{align*}

If we additionally assume that $\nu_p(6k-1)$ is odd whenever the prime $p$
has residue $2$ modulo $3$, then, for such a $p$, we have $6k-1\equiv
0\mod p$, and $2k-1\equiv -4k\mod p$. Suppose $p$ divides
$(3n(6k-1)+3m+2)$. Then $3m+2\equiv 0\mod p$, and $3m\equiv -2\mod p$.
Multiplying by $2k$, we obtain 
$$m\equiv -4k\equiv 2k-1\mod p,$$ 
which is a contradiction. Hence, $\nu_p(3N+1)$ is odd and $p\in\Pi_2$.
Theorem~\ref{thm:eisenstein} again gives $a_3(N)=0$.
\item Let $N\in\N$ be such that $a_3(N)=0$. By
Theorem~\ref{thm:eisenstein}, there is $p\in\Pi_2$ such that $\nu_p(3N+1)$
is odd. In fact, we can assume that $p\neq 2$, because there are at least
two distinct prime numbers with this property since $3N+1\equiv 1\mod 3$.
Thus, $p^{\nu_p(3N+1)}\equiv -1\mod 6$, and $p^{\nu_p(3N+1)}\geq 5$. It
follows that there exists $k\geq 1$ such that $p^{\nu_p(3N+1)}=6k-1$. Let
now $q\in\N$ be such that $3N+1=p^{\nu_p(3N+1)}q$. Then $q$ is not
divisible by $p$. Note that we must have $q\equiv 2\mod 3$ since
$3N+1\equiv 1\mod 3$. Assume $q>2$. Then $q\geq 5$, and there is $m\geq 1$
such that $q=3m+2$. Furthermore, $q\not\equiv 0\mod p$. The same
computation as above gives that $m\not\equiv 2k-1 \mod p$, and
$$N=\frac{(6k-1)(3m+2)-1}3=(6k-1)m+4k-1.$$
Thus $N$ is as in (ii) with $n=0$. Finally, assume that
$q=2$. We have $3N+1=2(6k-1)=2(6(k-1)+5)=4\cdot 3(k-1)+10$. Hence,
$$N=4(k-1)+\frac 1 3 (10\cdot 4^0-1),$$
and $N$ is as in (i) for $n=k-1$ and $m=1$.

Hence, the point (iii) of \cite[5.2]{Han} holds if we use the new assumptions
introduced in (ii).
\end{enumerate}
\bigskip

{\bf Acknowledgements.} The authors wish to thank the referee for their
constructive remarks and suggestions that improved the organization and
emphasis of the paper.

\bibliographystyle{abbrv}
\bibliography{references_19_1}

\end{document}